\documentclass[11pt,a4paper,twoside,enabledeprecatedfontcommands,final]{scrartcl}
\usepackage{a4wide}
\usepackage{amsfonts}
\usepackage{amsmath}
\usepackage{amssymb}
\usepackage{amsthm}
\usepackage[numbers]{natbib}
\usepackage{booktabs}
\usepackage{algorithm,algpseudocode}
\usepackage{cancel}
\usepackage{MnSymbol} 
\usepackage{hyperref}

\newcommand{\N}{\ensuremath{\mathbb{N}}}
\newcommand{\T}{\ensuremath{\mathbb{T}}}
\newcommand{\Z}{\ensuremath{\mathbb{Z}}}
\newcommand{\R}{\ensuremath{\mathbb{R}}}
\newcommand{\C}{\ensuremath{\mathbb{C}}}
\newcommand{\ii}{\textnormal{i}}
\newcommand{\e}{\textnormal{e}}
\newcommand{\ceil}[1]{\left\lceil#1\right\rceil}
\newcommand{\floor}[1]{\left\lfloor#1\right\rfloor}

\newcommand{\boldx}{{\ensuremath{\boldsymbol{x}}}}
\newcommand{\boldy}{{\ensuremath{\boldsymbol{y}}}}
\newcommand{\boldt}{{\ensuremath{\boldsymbol{t}}}}
\newcommand{\boldp}{{\ensuremath{\boldsymbol{p}}}}
\newcommand{\boldk}{{\ensuremath{\boldsymbol{k}}}}
\newcommand{\boldh}{{\ensuremath{\boldsymbol{h}}}}
\newcommand{\boldz}{{\ensuremath{\boldsymbol{z}}}}
\newcommand{\boldzero}{{\ensuremath{\boldsymbol{0}}}}
\newcommand{\boldone}{{\ensuremath{\boldsymbol{1}}}}
\newcommand{\boldgamma}{{\ensuremath{\boldsymbol{\gamma}}}}

\newtheorem{theorem}{Theorem}[section]
\newtheorem{lemma}[theorem]{Lemma}
\newtheorem{remark}[theorem]{Remark}
\newtheorem{definition}[theorem]{Definition}
\newtheorem{corollary}[theorem]{Corollary}

\makeatletter
\def\imod#1{\allowbreak\mkern10mu({\operator@font mod}\,\,#1)}
\makeatother

\algnewcommand\algorithmicforeach{\textbf{for each}}
\algdef{S}[FOR]{ForEach}[1]{\algorithmicforeach\ #1\ \algorithmicdo}

\numberwithin{equation}{section}

\newcommand{\OO}[1]{\mathcal{O}\left(#1\right)}

\title{Multiple Lattice Rules for Multivariate $L_\infty$ Approximation in the Worst-Case Setting}
\subtitle{With applications in tractability and sampling numbers}
\date{03. September 2019}
\author{Lutz K\"ammerer\footnotemark[1]
        }
        
        \hypersetup{
          pdfauthor=Lutz K\"{a}mmerer
          pdftitle={Multiple Lattice Rules for Multivariate $L_\infty$ Approximation in the Worst-Case Setting},
          pdfsubject={65T40, 42B05, 68Q17, 68Q25, 42B35, 65T50, 65Y20, 65D30, 65D32},
          pdfkeywords={approximation of multivariate periodic functions, trigonometric polynomials, lattice rule, multiple rank-1 lattice, fast Fourier transform},
          pdfcreator = {pdflatex},
          plainpages=false,
          pdfstartview=Fit,       
          pdfpagemode=UseOutlines, 
          bookmarksnumbered=true, 
          bookmarksopen=false,     
          bookmarksopenlevel=0,   
          colorlinks=true,       
          linkcolor=black,
          citecolor=black,
          urlcolor=black}

\begin{document}

\maketitle

\begin{abstract}\small
We develop a general framework for estimating the $L_\infty(\T^d)$ error for the approximation of multivariate periodic functions
belonging to specific reproducing kernel Hilbert spaces using approximants that are trigonometric polynomials computed from sampling values.
The used sampling schemes are suitable sets of rank-1 lattices that can be constructed in an extremely efficient way. Furthermore, 
the structure of the sampling schemes allows for fast Fourier transform (FFT) algorithms. We present and discuss
one FFT algorithm and analyze the worst case $L_\infty(\T^d)$ error for this specific approach.

Using this general result we work out very weak requirements on the reproducing kernel Hilbert spaces that allow
for a simple upper bound on the sampling numbers in terms of approximation numbers, where the approximation error
is measured in the $L_\infty(\T^d)$ norm. Tremendous advantages of this estimate are its pre-asymptotic validity as well as
its simplicity and its specification in full detail. It turns out, that approximation numbers and sampling numbers 
differ at most slightly. The occurring multiplicative gap does not depend on the spatial dimension $d$ and depends at most logarithmically on the number of used linear information or sampling values, respectively.
 
Moreover, we illustrate the capability of the new sampling method with the aid of specific highly popular source spaces, which yields that the suggested algorithm is nearly optimal from different points of view. For instance, we improve tractability results for the $L_\infty(\T^d)$ approximation for sampling methods and we achieve almost optimal sampling rates for functions of dominating mixed smoothness.

Great advantages of the suggested sampling method are the
constructive methods for determining sampling sets that guarantee the shown error bounds,
the simplicity of all necessary implementations, i.e., the implementation of the FFT and the construction of the sampling schemes,
and the extreme efficiency of all these algorithms. 

\medskip


\medskip


\end{abstract}

\footnotetext[1]{
  Chemnitz University of Technology, Faculty of Mathematics, 09107 Chemnitz, Germany\\
  lutz.kaemmerer@mathematik.tu-chemnitz.de, Phone:+49-371-531-37728, Fax:+49-371-531-837728
}

\section{Introduction}
During the last decades, a huge number of papers deal with the approximation of multivariate functions.
On the one hand there is the classical approach that searches for optimal algorithms
that use samples for computing approximants. On the other hand one is interested in 
optimal algorithms that use general linear information of the function in order to construct
an approximant.
Usually, one considers the relation of the approximation error, that is measured in a specific norm,
to the number of used information of the function.
In this context, optimality often means best possible in the sense of an worst case approximation error, i.e.,
the worst case of the above mentioned relation with respect to all functions belonging to the unit ball of a
specific source space.

In this paper, we consider periodic functions and we measure the worst case error in the $L_\infty$ norm, similar
to \cite{KuWaWo08,CoKueSi16}, where the authors are interested in best possible approximations based on linear information,
and \cite{KuWaWo09,KuWaWo09b,ByDuSiUl14,ByKaUlVo16}, where the sampling errors, i.e., the approximation errors based on sampling methods,
are considered.
The constructive approaches of the latter ones use sparse grids or single rank-1 lattices as sampling schemes for determining upper and lower bounds on the sampling errors
in terms of the number of used samples. Often, the bounds only allow for asymptotic statements, where the concrete dependencies on, e.g., the dimension
is not explicitly stated.
Exceptions are the papers \cite{KuWaWo09}, where the authors use single rank-1 lattices as sampling schemes, and \cite{KuWaWo09b}, where a completely non-constructive proof is used to show the existence of suitable sampling sets for tractability considerations. In both papers, the dependencies on all parameters are determined, i.e.,
the presented upper bounds can also be used in order to give error bounds even in pre-asymptotic settings.
Unfortunately, the errors decay with a rate that is far away from optimal ones when increasing the number of 
used samples. We stress on the fact that the known upper bounds based on the single rank-1 lattice approach can not be improved significantly, cf.~\cite{ByKaUlVo16}.

Another recent paper~\cite{KaVo19} analyzes a similar approximation method as we investigate in this paper. It also presents estimates on the corresponding $L_\infty$ errors. However, that paper focuses on specific function spaces
of generalized mixed smoothness, which can be treated as special cases of the spaces we study herein. Moreover, the dependencies
of the approximation errors on the number of sampling values is considered only in asymptotics, i.e., the detailed (exponential) dependencies on the dimension are missing and thus the results do not allow for suitable estimates in pre-asymptotic settings.

At this point, we stress that many papers presenting sampling errors associated with constructive designs of the used sampling schemes treat specific fixed function spaces and usually focus on dominating mixed smoothness spaces.

This paper deals with a newly developed sampling strategy that uses sampling schemes
that are unions of several rank-1 lattices.
We present an approximation algorithm and analyze the arising approximation error in full detail.
To this end, we consider suitable reproducing kernel Hilbert spaces and determine
the corresponding approximation error of our sampling strategy in terms of a so called
worst case truncation error, which is actually the best possible worst case approximation error
one can achieve using general linear information, cf.~\cite{CoKueSi16}.

The remainder of the paper is structured as follows.
Section~\ref{sec:pre} gives a short overview on the fundamentals of our considerations.
In Section~\ref{sec:general_framework}, we analyze the new approximation approach
and prove the general framework for estimating associated sampling errors.
It turns out that the sampling error is almost as good as the best possible worst case approximation error, i.e.,
in our setting the approximation computed from the sampling values is -- up to a factor that depends only logarithmically on the number of approximated Fourier coefficients -- identical to the error that occurs when using the exact Fourier partial sum for approximation.
Subsequently, we apply the result to a highly popular type of reproducing kernel Hilbert spaces, often called Korobov spaces or Sobolev spaces of dominating mixed smoothness, that
are most widely used as illustrating examples during tractability considerations, in Section~\ref{sec:tractability}.
We improve tractability results presented in~\cite{KuWaWo09} for lattice algorithms.
In addition, the result of this paper even improves the known upper bounds on the rates of convergence
for tractable $L_\infty$ approximation based on sampling values in general, cf.~\cite{KuWaWo09b}. In fact, the convergence rates proved in this paper corresponds -- up to an arbitrary small $\epsilon$ -- to
the best possible convergence rates that can be achieved by algorithms that use general linear information, cf.~\cite{KuWaWo08} for details. Remark~\ref{rem:tractability} discusses the substantial improvements presented in this paper -- in particular with respect to tractability.

The insight into the calculations that occurs for Korobov spaces leads directly to 
the requirements  that need to be fulfilled in order to show a strong
relation between approximation numbers and sampling numbers, which we discuss in Section~\ref{sec:app_samp_numbers}.
Naturally, approximation numbers are bounded from above by sampling numbers.
Under mild assumptions on the considered function spaces, our general result from Section~\ref{sec:general_framework}
can be applied in order to estimate sampling numbers in terms of approximation numbers.
It turns out that the $L_\infty$ sampling numbers and approximation numbers may differ at most slightly,
cf.~\cite{CoKueSi16}. In more detail, the main rate of sampling and approximation numbers with respect to
the number of used information remain the same.

We stress the fact that all suggested algorithms, i.e., the algorithm for computing the approximant, cf. Algorithm~\ref{alg:compute_S_I_Lambda}, as well as the algorithm that determines suitable sampling schemes, cf. e.g. \cite[Algorithms~1 \&~3]{Kae17}, can be extremely efficiently performed with respect to the used arithmetic operations and thus offer a reasonable sampling strategy for practical applications.
On a final note, we would like to point out that the very cheap random construction of the suggested sampling sets
may fail with a certain small probability, but checking these requirements has the same complexity as the
suggested fast Fourier transform algorithm and, thus, is extremely efficient, cf.~\cite{Kae17}.

\section{Prerequisites}\label{sec:pre}

\subsection{Reproducing kernel Hilbert spaces}\label{sec:repro_kernel_hilbert_space}

In order to apply sampling strategies, we consider continuous periodic functions $f\,\colon\,\T^d\to \C$, $\T\sim[0,1)$, denote their Fourier coefficients by
\begin{align}
c_\boldh(f):=\int_{\T^d}f(\boldx)e^{-2\pi\ii\boldh\cdot\boldx}\mathrm{d}\boldx,
\end{align}
and think of the function $f$ as a Fourier series
$$
f(\boldx):=\sum_{\boldh\in\Z^d}c_\boldh(f)\e^{2\pi\ii\boldh\cdot\boldx},
$$
where $\boldh\cdot\boldx=\sum_{j=1}^dh_jx_j$ is the usual inner product in $\R^d$.
Furthermore, the function spaces under consideration are reproducing kernel Hilbert spaces, where we assume that the reproducing kernel $K_d\,\colon\,\T^d\times\T^d\to\C$ is given by
\begin{align*}
K_d(\boldx,\boldy):=\sum_{\boldh\in\Z^d}\frac{\e^{2\pi\ii\boldh\cdot(\boldx-\boldy)}}{r_d(\boldh)}.
\end{align*}
The occurring weight function 
$r_d\,\colon\,\Z^d\to(0,\infty)$ is subject to the restriction that
 $$\sum_{\boldh\in\Z^d}r_d(\boldh)^{-1}<\infty$$ holds, which guarantees the continuity of the positive definite kernel $K_d$.

Due to \cite{Aron50}, the positive definite kernel $K_d$ is indeed an reproducing kernel and it induces an inner product, i.e., $f(\boldy)=\langle f,K_d(\circ,\boldy)\rangle_d$ for all appropriate functions $f$, which is given by
\begin{align}
\langle f,g\rangle_d:=\sum_{\boldh\in\Z^d}c_\boldh(f)\overline{c_\boldh(g)}r_d(\boldh).
\label{eq:scalar_product}
\end{align}
The associated norm $\|f|\mathcal{H}_r(\T^d)\|:=\sqrt{\langle f,f\rangle_d}=\left(\sum_{\boldk\in\Z^d}r_d(\boldh)|c_\boldh(f)|^2\right)^{1/2}$ directly leads to the reproducing  kernel Hilbert space
\begin{align*}
\mathcal{H}_r(\T^d):=\left\{f\in L_1(\T^d)\,\colon\,\|f|\mathcal{H}_r(\T^d)\|<\infty\right\}
\end{align*}
of all functions $f$ for which the norm is finite.

\subsection{Multiple Rank-1 lattices}
Recently, a spatial discretization approach for multivariate trigonometric polynomials was presented in \cite{Kae17},
which we will utilize in order to compute approximations based on sampling values. To this end, we define
a rank-1 lattice
$$
\Lambda(\boldz,M):=\left\{\frac{j}{M}\boldz\bmod{\boldone}\,\colon\,j=0,\ldots,M-1\right\}\subset\T^d,
$$
where $M\in\N$ is called lattice size and $\boldz\in\Z^d$ is the generating vector of the rank-1 lattice.

One main advantage of rank-1 lattices is the group structure of the sampling set, which allows for fast Fourier transform algorithms, cf.~\cite{kaemmererdiss}. At the same time, this structure is the main disadvantage of a single rank-1 lattice since this structure causes excessive oversampling factors for spatial discretizations of specific trigonometric polynomials, cf. \cite[Chapter~3]{kaemmererdiss}, and -- as a consequence -- sampling rates that are far away from optimal ones, cf.~\cite{ByKaUlVo16}.

In order to avoid these disadvantages, we consider sampling sets that are unions of several rank-1 lattices
$$
\Lambda:=\Lambda(\boldz_1,M_1,\ldots,\boldz_L,M_L):=\bigcup_{\ell=1}^L\Lambda(\boldz_\ell,M_\ell),
$$
which we call \emph{multiple rank-1 lattice}. Our considerations are essentially based on the observation that
a multiple rank-1 lattice that fulfills the equality
\begin{align}
I=\bigcup_{\ell=1}^L\tilde{I}_\ell,\label{eq:reco_prop}
\end{align}
where $\tilde{I}_\ell:=\left\{\boldk\in I\,\colon\,\boldk\cdot\boldz_\ell\not\equiv\boldh\cdot\boldz_\ell\imod{M_\ell}\text{ for all }\boldh\in I\setminus\{\boldk\}\right\}$ depends on the rank-1 lattice  $\Lambda(\boldz_\ell,M_\ell)$,
is necessarily a spatial discretization for all trigonometric polynomials with frequency support in $I$, cf.~\cite{Kae16, Kae17} for details.
Moreover, we stress the fact that the number of sampling nodes within the multiple rank-1 lattice $\Lambda$ is bounded by
$|\Lambda|\le1+\sum_{\ell=1}^L(M_\ell-1)\le\sum_{\ell=1}^L M_\ell$. Under mild assumptions, $M_\ell$ can be chosen such that $M_\ell\lesssim|I|$ holds. In addition, the number $L$ of used rank-1 lattices can be bounded by $L\lesssim\log|I|$ in order to construct multiple rank-1 lattices that fulfill the reconstruction property in \eqref{eq:reco_prop}, cf.~\cite{Kae17}.
In this work, we equivalently modify the reconstruction property \eqref{eq:reco_prop} in order to simplify the theoretical considerations in Section~\ref{sec:general_framework}. We refer to Remark~\ref{rem:averaging_approx_fc}
for a detailed discussion on different reconstruction algorithms that could be used for approximation.

\section{General framework}\label{sec:general_framework}

The definition of
\begin{align}
I_\ell:=\left\{\boldh\in I\setminus\bigcupdot_{j=1}^{\ell-1} I_\ell\colon \boldh\cdot\boldz_\ell\not\equiv \boldk\cdot\boldz_\ell\imod{M_\ell}\text{ for all }\boldk\in I\setminus\{\boldh\}\right\}\label{eq:def_Iell}
\end{align}
guarantees that the sets $I_\ell$, $\ell=1,\ldots,L$, are disjoint.

For fixed $\boldh\in\bigcupdot_{\ell=1}^{L}I_\ell$ there is exactly one $l_\boldh$ for which $\boldh\in I_{\ell_\boldh}$ holds, i.e., the number
\begin{align}
\ell_\boldh\in\left\{\ell\in\{1,\ldots,L\}\colon \boldh\in I_\ell\right\}\label{eq:def_lh}
\end{align}
is already uniquely and well defined by \eqref{eq:def_lh}. A clearer style for defining $\ell_\boldh$ is
\begin{equation}
\ell_\boldh:=\max\left\{\ell\in\{1,\ldots,L\}\colon \boldh\in I_\ell\right\}=\min\left\{\ell\in\{1,\ldots,L\}\colon \boldh\in I_\ell\right\}.
\label{eq:def_lh1}
\end{equation}

\begin{algorithm}[tb]
\caption{Approximation of Fourier coefficients using multiple lattice rules.}\label{alg:compute_S_I_Lambda}
  \begin{tabular}{p{2.25cm}p{5cm}p{7cm}}
    Input: 	& $I\subset\Z^d$ 										&frequency set\\
    		& $\Lambda:=\Lambda(\boldz_1,M_1,\ldots,\boldz_L,M_L)$ 	& sampling nodes\\
	   		& $\left\{f(\boldx)\colon\boldx\in\Lambda\right\}$		& sampling values of the function $f$ \\
  \end{tabular}
  \begin{algorithmic}[1]
  \State Set $\tilde{I}=\emptyset$.
	\For{$\ell=1 \textnormal{ to }L$}
		\State Compute $\hat{g}_t^{(\ell)}:=\sum_{j=0}^{M_\ell-1}p\left(\frac{j}{M_\ell}\boldz\right)\e^{-2\pi i\frac{jt}{M_\ell}}$, $t=0,\ldots,M_\ell-1$, using a 1d FFT.
		\State Determine $I_\ell:=\left\{\boldk\in I\setminus\tilde{I}\;\colon\; \boldk\cdot\boldz_\ell\not\equiv \boldh\cdot\boldz_\ell\imod{M_\ell}\text{ for all }\boldh\in I\setminus\{\boldk\}\right\}$
		\ForEach{$\boldh\in I_\ell$}
			\State	Set $\hat{f}_\boldh:=\hat{g}_{\boldh\cdot\boldz_\ell\bmod{M_\ell}}^{(\ell)}$.
		\EndFor
		\State Set $\tilde{I}:=\tilde{I}\cupdot I_\ell$.
  	\EndFor
	\end{algorithmic}
	\begin{tabular}{p{2.25cm}p{3cm}p{9cm}}
	Output: & $\tilde{I}\subset I$								& frequencies of uniquely reconstructable\\ &&Fourier coefficients for $f\in\Pi_I$ and\\
	& $\{\hat{f}_\boldh\}_{\boldh\in\tilde{I}}$			& corresponding approximated Fourier coefficients\\
    \cmidrule{1-3}
	  Complexity: & \multicolumn{2}{p{12cm}}{$\OO{L(\tilde{M}\log \tilde{M}+|I|(d+\log|I|))}$, where $\tilde{M}:=\max\{M_\ell\,\colon\,\ell=1,\ldots,L\}$}
	\end{tabular}
\end{algorithm}

We compute the approximation of a function $f\in \mathcal{H}_r(\T^d)$ using Algorithm \ref{alg:compute_S_I_Lambda} and get
the approximant
$$
S_I^\Lambda f(\boldx)=\sum_{\ell=1}^{L}\sum_{\boldh\in I_\ell}\hat{f}_\boldh\e^{2\pi\ii\boldh\cdot\boldx},
$$
where the approximated Fourier coefficients are computed by
\begin{align}
\hat{f}_\boldh:=\frac{1}{M_{\ell_\boldh}}\sum_{j=0}^{M_{\ell_\boldh}-1}f\left(\frac{j\boldz_{\ell_\boldh}}{M_{\ell_\boldh}}\right)\e^{-2\pi\ii j\frac{\boldh\cdot\boldz_{\ell_\boldh}}{M_{\ell_\boldh}}}
=\sum_{\substack{\boldk\in\Z^d\\\boldk\cdot\boldz_{\ell_\boldh}\equiv\boldh\cdot\boldz_{\ell_\boldh}\imod{M_{\ell_\boldh}}}}c_\boldk(f).\label{eq:alias_Lambda_l}
\end{align}
Using an efficient sort algorithm for determining the frequency sets $I_\ell$ and one-dimensional fast Fourier transforms of lengths $M_\ell$, $\ell=1,\ldots,L$, for computing all the numbers $\hat{g}_t^{(\ell)}$ leads to the arithmetic complexity $\OO{L|I|(d+\log|I|)+\sum_{\ell=1}^LM_\ell\log M_\ell}$ of Algorithm~\ref{alg:compute_S_I_Lambda}, cf.~\cite{Kae16}.

The equality in \eqref{eq:alias_Lambda_l} holds due to the well known aliasing properties of rank-1 lattices, cf.\ \cite[Theorem~2.8]{SlJo94} and \cite[Section~3.4]{kaemmererdiss}.
At this point, we stress that we observe $S_I^\Lambda f\in\Pi_{\tilde{I}}\subseteq \Pi_I$, $\tilde{I}:=\bigcupdot_{\ell=1}^L I_\ell$, and 
that $\tilde{I}\subsetneq I$ might hold, in general.
In the following, we assume $\tilde{I}=I$, which actually is the crucial characteristic of the used multiple rank\mbox{-}1 lattice $\Lambda(\boldz_1,M_1,\ldots,\boldz_L,M_L)$. This assumption will prove to be very beneficial in the following considerations of the pointwise error of a function $f\in\mathcal{H}_r(\T^d)$ and its approximation $S_I^\Lambda f$. We determine
\begin{align}
(f&-\operatorname{S}_I^\Lambda f)(\boldx)\nonumber\\
&=\sum_{h\not\in I}c_\boldh(f)\e^{2\pi\ii\boldh\cdot\boldx}+\underbrace{\sum_{\boldh\in I}\left(c_\boldh(f)-\frac{1}{M_{\ell_\boldh}}\sum_{j=0}^{M_{\ell_\boldh}-1}f\left(\frac{j\boldz_{\ell_\boldh}}{M_{\ell_\boldh}}\right)\e^{-2\pi\ii j\boldh\cdot\boldz_{\ell_\boldh}/M_{\ell_\boldh}}\right)\e^{2\pi\ii\boldh\cdot\boldx}}_{=:\operatorname{R}_I^\Lambda f}\label{eq:aliasing_formula}.
\end{align}

The first summand in \eqref{eq:aliasing_formula} is called the \emph{truncation error} and is inevitable when approximating the function $f$ by a trigonometric polynomial with frequencies supported on the set $I$.
Thus the main focus is on estimating the second summand in \eqref{eq:aliasing_formula}, which is denoted by $\operatorname{R}_I^\Lambda f$. For that reason, we consider the Fourier coefficients of the trigonometric polynomial $\operatorname{R}_I^\Lambda f$. We follow the considerations in \cite{KuSlWo06} and observe
\begin{align}
&c_\boldh(f)-\frac{1}{M_{\ell_\boldh}}\sum_{j=0}^{M_{\ell_\boldh}-1}f\left(\frac{j\boldz_{\ell_\boldh}}{M_{\ell_\boldh}}\right)\e^{-2\pi\ii j\boldh\cdot\boldz_{\ell_\boldh}/M_{\ell_\boldh}}=\langle f,\tau_\boldh\rangle_d,\nonumber
\intertext{where $\tau_\boldh$, $\boldh\in I$, is defined by}
\tau_\boldh(\boldt)&:=\int_{\T^d}K_d(\boldt,\boldx)\e^{2\pi\ii\boldh\cdot\boldx}\mathrm{d}\boldx-\frac{1}{M_{\ell_\boldh}}\sum_{j=0}^{M_{\ell_\boldh}-1}K_d\left(\boldt,\frac{j\boldz_{\ell_\boldh}}{M_{\ell_\boldh}}\right)\e^{2\pi\ii j\boldh\cdot\boldz_{\ell_\boldh}/M_{\ell_\boldh}}\nonumber
\intertext{since $K_d$ is the reproducing kernel. Exploiting the linearity of the scalar product $\langle\circ,\circ\rangle_d$, cf. \eqref{eq:scalar_product}, yields}
\tau_\boldh(\boldt)&=-\sum_{\substack{\boldk\in\Z^d\setminus\{\boldh\}\\\boldh\cdot\boldz_{\ell_\boldh}\equiv\boldk\cdot\boldz_{\ell_\boldh}\imod{M_{\ell_\boldh}}}}\frac{\e^{2\pi\ii\boldk\cdot\boldt}}{r_d(\boldk)}\label{eq:tau_h_in_I}
\end{align}
for all $\boldh\in I$.
For fixed $\boldx$, we apply the Cauchy--Schwarz inequality and achieve the estimate
\begin{align}
|(f-\operatorname{S}_I^\Lambda f)(\boldx)|=
\left|
\left\langle
f,\sum_{\boldh\in\Z^d}\tau_\boldh \e^{2\pi\ii\boldh\cdot\boldx}
\right\rangle_d
\right|
\le\left\|f\,|\mathcal{H}_r(\T^d)\right\|\left\|\sum_{\boldh\in\Z^d}\tau_\boldh\e^{2\pi\ii\boldh\cdot\boldx}\,\bigg.\bigg|\mathcal{H}_r(\T^d)\right\|
\label{eq:wce_derivation}
\end{align}
where the functions $\tau_\boldh$  are defined by
\begin{align*}
\tau_\boldh(\boldt)&=
\begin{cases}
\text{given in \eqref{eq:tau_h_in_I}}, & \boldh\in I,\\
\int_{\T^d}K_d(\boldt,\boldx)\e^{2\pi\ii\boldh\cdot\boldx}\mathrm{d}
\boldx=\frac{\e^{2\pi\ii\boldh\cdot\boldt}}{r_d(\boldh)}, & \boldh\in \Z^d\setminus I.
\end{cases}
\end{align*}
We stress that each $\tau_\boldh$ depends on the spatial variable $\boldt$ and that the norm in \eqref{eq:wce_derivation} is taking with respect to this variable.

Taking \eqref{eq:wce_derivation} into account, the worst case error measured in the $L_\infty(\T^d)$ norm of a function $f\in\mathcal{H}_r(\T^d)$ within the unit ball of $\mathcal{H}_r(\T^d)$ is bounded from above by
\begin{align*}
\sup_{\|f|\mathcal{H}_r(\T^d)\|\le 1}\|f-\operatorname{S}_I^\Lambda f|L_\infty(\T^d)\|
\le\sup_{\boldx\in\T^d}\left\|\sum_{\boldh\in\Z^d}\tau_\boldh\e^{2\pi\ii\boldh\cdot\boldx}\bigg.\bigg|\mathcal{H}_r(\T^d)\right\|.
\end{align*}
For ease of notation, we write $\boldh\not\in I$ for $\boldh\in\Z^d\setminus I$ in the following.
We estimate
\begin{align}
\sup_{\|f|\mathcal{H}_r(\T^d)\|\le 1}&\|f-\operatorname{S}_I^\Lambda f|L_\infty(\T^d)\|^2
\le\sup_{\boldx\in\T^d}\left|\sum_{\boldh\in\Z^d}\sum_{\boldk\in\Z^d}\langle\tau_\boldh,\tau_\boldk\rangle_d\e^{2\pi\ii(\boldh+\boldk)\cdot\boldx}\right|\nonumber\\
&=
\sup_{\boldx\in\T^d}\Bigg|
\sum_{\boldh\not\in I}\sum_{\boldk\not\in I}\langle\tau_\boldh,\tau_\boldk\rangle_d\e^{2\pi\ii(\boldh+\boldk)\cdot\boldx}
+\sum_{\boldh\in I}\sum_{\boldk\in I}\langle\tau_\boldh,\tau_\boldk\rangle_d\e^{2\pi\ii(\boldh+\boldk)\cdot\boldx}\nonumber\\
&\quad+\sum_{\boldh\not\in I}\sum_{\boldk\in I}\langle\tau_\boldh,\tau_\boldk\rangle_d\e^{2\pi\ii(\boldh+\boldk)\cdot\boldx}
+\sum_{\boldh\in I}\sum_{\boldk\not\in I}\langle\tau_\boldh,\tau_\boldk\rangle_d\e^{2\pi\ii(\boldh+\boldk)\cdot\boldx}
\Bigg|\nonumber\\
&\le
\underbrace{\sum_{\boldh\not\in I}\sum_{\boldk\not\in I}|\langle\tau_\boldh,\tau_\boldk\rangle_d|}_{=:\Sigma_{\bcancel{I}\bcancel{I}}}
+2\underbrace{\sum_{\boldh\not\in I}\sum_{\boldk\in I}|\langle\tau_\boldh,\tau_\boldk\rangle_d|}_{=:\Sigma_{\bcancel{I}I}}
+\underbrace{\sum_{\boldh\in I}\sum_{\boldk\in I}|\langle\tau_\boldh,\tau_\boldk\rangle_d|}_{=:\Sigma_{I I}}
\label{eq:split_wceinfty}.
\end{align}
We individually treat the summands $\Sigma_{\bcancel{I}\bcancel{I}}$, $\Sigma_{\bcancel{I}I}$, and $\Sigma_{I I}$
and start with the simplest case.
Formula \eqref{eq:scalar_product} implies
\begin{align*}
\langle\tau_\boldh,\tau_\boldk\rangle_d=\begin{cases}
0,&\boldh\neq\boldk\\
r_d(\boldh)^{-1},&\boldh=\boldk
\end{cases}
\end{align*}
for $\boldh\not\in I$ and $\boldk\not\in I$, which directly yields
\begin{align}
\Sigma_{\bcancel{I}\bcancel{I}}=
\sum_{\boldh\not\in I}\sum_{\boldk\not\in I}|\langle\tau_\boldh,\tau_\boldk\rangle_d|
=\sum_{\boldh\in\Z^d\setminus I}r_d(\boldh)^{-1}.\label{eq:trunc_error}
\end{align}

We call the term in \eqref{eq:trunc_error} the \emph{worst case truncation error}.
The strategy of the following considerations is to
estimate the terms $\Sigma_{\bcancel{I}I}$ and $\Sigma_{I I}$
by terms in the worst case truncation error $\Sigma_{\bcancel{I}\bcancel{I}}$.
To this end, we analyze \eqref{eq:tau_h_in_I} in more detail using
specific Kronecker delta functions.
\begin{definition}
For each $\boldk\in\Z^d$, $\Lambda(\boldz_1,M_1,\ldots,\boldz_L,M_L)$ a multiple rank-1 lattice and $\ell\in\{1,\ldots,L\}$, we define the Kronecker delta functions $\delta_\boldk^{(\ell)}\colon \Z^d\to\{0,1\}$ by
$$\delta_\boldk^{(\ell)}(\boldh)=
\begin{cases}
1, & \boldk\in I_\ell,\;\boldk\neq\boldh\in\Z^d,\;\text{and }\boldh\cdot\boldz_\ell\equiv\boldk\cdot\boldz_\ell\imod{M_\ell};\\
0, & \text{otherwise}.
\end{cases}
$$
\end{definition}
We determine some helpful properties of the introduced Kronecker delta functions.
\begin{lemma}\label{lem:properties_of_delta}
Let a frequency index set $I\subset\Z^d$, $|I|<\infty$, and a multiple rank-1 lattice $\Lambda(\boldz_1,M_1,\ldots,\boldz_L,M_L)$ be given.
The frequency sets $I_\ell$, $\ell=1,\ldots,L$, are determined as specified in \eqref{eq:def_Iell}.
Then the following hold.
\begin{itemize}
\item
For each $\boldk\in I_\ell$, we characterize the frequency set of aliasing Fourier coefficients by
\begin{align}
\left\{\boldh\in\Z^d\setminus\{\boldk\}\colon\boldh\cdot\boldz_\ell\equiv\boldk\cdot\boldz_\ell\imod{M_\ell}\right\}
=\{\boldh\in\Z^d\colon
\delta_\boldk^{(\ell)}(\boldh)=1\},\label{eq:aliasing_set_delta}
\end{align}
using the above defined Kronecker delta functions $\delta_\boldk^{(\ell)}$.
\item
Furthermore, the equality 
\begin{equation}
\{\boldh\in\Z^d\colon
\delta_\boldk^{(\ell)}(\boldh)=1\}=\emptyset\quad\text{for each}\quad
\boldk\in\Z^d\setminus I_\ell\label{equ:empty_aliasing}
\end{equation}
holds.
\item
For
fixed $\boldk\in\Z^d$ and fixed $\ell\in\{1,\ldots,L\}$ we observe
\begin{align}
\{\boldh\in\Z^d\colon
\delta_\boldk^{(\ell)}(\boldh)=1\}\cap I=\emptyset.
\label{eq:aliasing_within_I}
\end{align}
\item
Moreover, for each fixed $\ell\in\{1,\ldots,L\}$ and each fixed $\boldh\in\Z^d$ we have
\begin{align}
\sum_{\boldk\in I_\ell}\delta^{(\ell)}_\boldk(\boldh)\in\{0,1\},\label{eq:aliasing_summation}
\end{align}
\item
which implies for fixed $\boldh\in\Z^d$
\begin{align}
0\le\sum_{\boldk\in I}\sum_{\ell=1}^L\delta^{(\ell)}_\boldk(\boldh)=\sum_{\boldk\in \bigcupdot_{\ell=1}^L I_\ell}\delta_\boldk^{(\ell_\boldk)}(\boldh)=\sum_{\ell=1}^L\sum_{\boldk\in I_\ell}\delta^{(\ell)}_\boldk(\boldh)\le L\label{eq:aliasing_summation_L}.
\end{align}
\end{itemize}

\end{lemma}

\begin{proof}
For $\boldk\in\Z^d\setminus I_\ell$ the Kronecker delta function $\delta_\boldk^{(\ell)}$ maps to zero for each $\boldh\in\Z^d$. Accordingly, we observe \eqref{equ:empty_aliasing}.

On the other hand, for $\boldk\in I_\ell$ the function value $\delta_\boldk^{(\ell)}(\boldh)$ is one exactly for those $\boldh\in\Z^d\setminus\{\boldk\}$, that fulfill the aliasing formula for $\Lambda(\boldz_\ell,M_\ell)$, namely 
$\boldh\cdot\boldz_\ell\equiv\boldk\cdot\boldz_\ell\imod{M_\ell}$, which characterizes the set at the left hand side in \eqref{eq:aliasing_set_delta}.

Due to the definition of the set
$$
I_\ell:=\left\{\boldh\in I\setminus\bigcupdot_{j=1}^{\ell-1} I_\ell\colon \boldh\cdot\boldz_\ell\not\equiv \boldk\cdot\boldz_\ell\imod{M_\ell}\text{ for all }\boldk\in I\setminus\{\boldh\}\right\}
$$
each $\boldk\in I_\ell$ has no aliasing element in $I$. Accordingly, for $\boldk\in I_\ell$ the equality in \eqref{eq:aliasing_within_I} holds. For $\boldk\in\Z^d\setminus I_\ell$ the equality also holds, since the set characterized by the Kronecker delta functions is already the empty set.

In order to prove \eqref{eq:aliasing_summation}, we have to show that each $\boldh\in\Z^d$ aliases to at most one
$\boldk\in I_\ell$. 
We assume the contrary, i.e., 
assume there exists $\boldh\in\Z^d$ such that there are $\boldk,\boldk'\in I_\ell\subset I$, $\boldk\neq\boldk'$, with $\delta_\boldk^{(\ell)}(\boldh)=1=\delta_{\boldk'}^{(\ell)}(\boldh)$, which implies $\boldk\neq\boldh\neq\boldk'$ and 
$\boldk\cdot\boldz_\ell\equiv\boldh\cdot\boldz_\ell\equiv\boldk'\cdot\boldz_\ell\imod{M_\ell}$. Accordingly, there is an aliasing element $\boldk'\in I$ for $\boldk\in I_\ell$,
which prohibits $\boldk$ from belonging to $I_\ell$, due to its definition. Thus, we observe the contradiction to our assumptions. Consequently, for each $\boldh\in\Z^d$ at most one $\delta_\boldk^{(\ell)}(\boldh)$, $\boldk\in I_\ell$ and $\ell$ fixed, can be nonzero. Moreover, for $|I_\ell|<M_\ell$ there exist $M_\ell-|I_\ell|$ of the disjoint sets
$$\{\boldh\in\Z^d\colon j\equiv\boldh\cdot\boldz_\ell\imod{M_\ell}\}\qquad j\in\{0,\ldots,M_\ell-1\}$$
that do not contain an element from $I_\ell$. For elements $\boldh\in\Z^d$ of these specific sets we observe
$\sum_{\boldk\in I_\ell}\delta^{(\ell)}_\boldk(\boldh)=0$.
According to that, the term $\sum_{\boldk\in I_\ell}\delta^{(\ell)}_\boldk(\boldh)$ is a non-negative integer of at least zero and at most one.
\newline
Finally, we prove \eqref{eq:aliasing_summation_L}.
For $\boldk\in I\setminus\bigcupdot_{\ell=1}^L I_\ell$
we observe $\delta^{(\ell)}_\boldk(\boldh)=0$ for all $\ell\in\{1,\ldots,L\}$, for $\boldk\in \bigcupdot_{\ell=1}^L I_\ell$, we have 
$\delta^{(\ell)}_\boldk(\boldh)=0$ for all $\ell\in\{1,\ldots,L\}\setminus\{\ell_\boldk\}$, which
justifies the first equality. The same observation yields the second equality
\begin{align*}
\sum_{\boldk\in I}\sum_{j=1}^L\delta_\boldk^{(j)}(\boldh)
=\sum_{\ell=1}^L\sum_{\boldk\in I_\ell}\sum_{j=1}^L\delta_\boldk^{(j)}(\boldh)
=\sum_{\ell=1}^L\sum_{\boldk\in I_\ell}\delta_\boldk^{(\ell)}(\boldh).
\end{align*}
The inequalities follow from \eqref{eq:aliasing_summation}.
\end{proof}
The introduced Kronecker delta functions $\delta_\boldk^{(\ell)}$ allow for concise
characterizations of the aliasing effects of the sampling method under consideration.
We exploit the observations of Lemma~\ref{lem:properties_of_delta} in order to estimate
the terms $\Sigma_{\bcancel{I}I}$ and $\Sigma_{II}$.
\begin{lemma}\label{lem:est_sum_cI_I}
Let $\Lambda(\boldz_1,M_1,\ldots,\boldz_L,M_L)$ be a multiple rank-1 lattice
such that $I=\bigcupdot_{\ell=1}^{L}I_\ell$, where $I_\ell$ is defined in \eqref{eq:def_Iell}.
Then we have
$$\Sigma_{\bcancel{I}I}\le L\Sigma_{\bcancel{I}\bcancel{I}}.$$
\end{lemma}

\begin{proof}
Let $\boldh\not\in I$, $\boldk\in I$ be given. Due to $I=\bigcupdot_{\ell=1}^{L}I_\ell$ 
there exists a unique $\ell_\boldk$, cf. \eqref{eq:def_lh} and \eqref{eq:def_lh1}, such that $\boldk\in I_{\ell_\boldk}$.
Accordingly, we observe
\begin{align}
\langle\tau_\boldh,\tau_\boldk\rangle_d&=
\left\langle
\frac{\e^{2\pi\ii\boldh\cdot\circ}}{r_d(\boldh)},
-\sum_{\substack{\boldk'\in\Z^d\setminus\{\boldk\}\\\boldk'\cdot\boldz_{\ell_\boldk}\equiv\boldk\cdot\boldz_{\ell_\boldk}\imod{M_{\ell_\boldk}}}}
\frac{\e^{2\pi\ii\boldk'\cdot\circ}}{r_d(\boldk')}
\right\rangle_d\nonumber\\
&=
-\sum_{\substack{\boldk'\in\Z^d\setminus\{\boldk\}\\\boldk'\cdot\boldz_{\ell_\boldk}\equiv\boldk\cdot\boldz_{\ell_\boldk}\imod{M_{\ell_\boldk}}}}
\left\langle
\frac{\e^{2\pi\ii\boldh\cdot\circ}}{r_d(\boldh)},
\frac{\e^{2\pi\ii\boldk'\cdot\circ}}{r_d(\boldk')}
\right\rangle_d
=
-\frac{\delta_{\boldk}^{({\ell_\boldk})}(\boldh)}{r_d(\boldh)}\label{eq:sp_InI}.
\end{align}
Consequently, we have
\begin{align*}
\Sigma_{\bcancel{I}I}&:=
\sum_{\boldh\not\in I}\sum_{\boldk\in I}|\langle\tau_\boldh,\tau_\boldk\rangle_d|=
\sum_{\boldh\not\in I}r_d(\boldh)^{-1}\sum_{\ell=1}^L\sum_{\boldk\in I_\ell}\delta_{\boldk}^{(\ell)}(\boldh)
\stackrel{\eqref{eq:aliasing_summation_L}}{\le}
L\Sigma_{\bcancel{I}\bcancel{I}}
\end{align*}
\end{proof}

\begin{lemma}\label{lem:est_sum_I_I}
Let $\Lambda(\boldz_1,M_1,\ldots,\boldz_L,M_L)$ be a multiple rank-1 lattice
such that $I=\bigcupdot_{\ell=1}^{L}I_\ell$, where $I_\ell$ is defined in \eqref{eq:def_Iell}.
Then we have
$$\Sigma_{I I}\le L^2\Sigma_{\bcancel{I}\bcancel{I}}.$$
\end{lemma}

\begin{proof}
Let $\boldh,\boldk\in I$ be given.
Due to $I=\bigcupdot_{\ell=1}^{L}I_\ell$ 
there exist unique $\ell_\boldk$, $\ell_\boldh$, cf. \eqref{eq:def_lh} and \eqref{eq:def_lh1}, such that $\boldk\in I_{\ell_\boldk}$ and 
$\boldh\in I_{\ell_\boldh}$.
We observe for a single scalar product
\begin{align*}
\langle\tau_\boldh,\tau_\boldk\rangle_d&=
\left\langle
-\sum_{\substack{\boldh'\in\Z^d\setminus\{\boldh\}\\\boldh'\cdot\boldz_{\ell_\boldh}\equiv\boldh\cdot\boldz_{\ell_\boldh}\imod{M_{\ell_\boldh}}}}
\frac{\e^{2\pi\ii\boldh'\cdot\circ}}{r_d(\boldh')}
,
-\sum_{\substack{\boldk'\in\Z^d\setminus\{\boldk\}\\\boldk'\cdot\boldz_{\ell_\boldk}\equiv\boldk\cdot\boldz_{\ell_\boldk}\imod{M_{\ell_\boldk}}}}
\frac{\e^{2\pi\ii\boldk'\cdot\circ}}{r_d(\boldk')}
\right\rangle_d
\\
&=
\sum_{\substack{\boldh'\in\Z^d\setminus\{\boldh\}\\\boldh'\cdot\boldz_{\ell_\boldh}\equiv\boldh\cdot\boldz_{\ell_\boldh}\imod{M_{\ell_\boldh}}}}
\left\langle
\frac{\e^{2\pi\ii\boldh'\cdot\circ}}{r_d(\boldh')}
,
\sum_{\substack{\boldk'\in\Z^d\setminus\{\boldk\}\\\boldk'\cdot\boldz_{\ell_\boldk}\equiv\boldk\cdot\boldz_{\ell_\boldk}\imod{M_{\ell_\boldk}}}}
\frac{\e^{2\pi\ii\boldk'\cdot\circ}}{r_d(\boldk')}
\right\rangle_d.
\intertext{Since we have $\boldh'\in\Z^d\setminus I$ for each $\boldh'$ in the equality above, we apply \eqref{eq:sp_InI} to each of the summands and achieve}
\langle\tau_\boldh,\tau_\boldk\rangle_d&=
\sum_{\substack{\boldh'\in\Z^d\setminus\{\boldh\}\\\boldh'\cdot\boldz_{\ell_\boldh}\equiv\boldh\cdot\boldz_{\ell_\boldh}\imod{M_{\ell_\boldh}}}}
\frac{\delta_{\boldk}^{({\ell_\boldk})}(\boldh')}{r_d(\boldh')}\\
&=
\sum_{\substack{\boldp\in\Z^d}}
\frac{\delta_{\boldh}^{({\ell_\boldh})}(\boldp)\,\delta_{\boldk}^{({\ell_\boldk})}(\boldp)}{r_d(\boldp)}
=
\sum_{\substack{\boldp\not\in I}}
\frac{\delta_{\boldh}^{({\ell_\boldh})}(\boldp)\,\delta_{\boldk}^{({\ell_\boldk})}(\boldp)}{r_d(\boldp)},
\end{align*}
where the last equality holds due to the equality $\delta_\boldh^{(\ell_\boldh)}(\boldp)=0$ for $\boldp\in I$.
Summing up these terms yields
\begin{align*}
\Sigma_{I I}:&=\sum_{\boldh\in I}\sum_{\boldk\in I}|\langle\tau_\boldh,\tau_\boldk\rangle_d|=
\sum_{\boldh\in I}\sum_{\boldk\in I}\langle\tau_\boldh,\tau_\boldk\rangle_d
=\sum_{\boldh\in I}\sum_{\boldk\in I}
\sum_{\substack{\boldp\not\in I}}
\frac{\delta_{\boldh}^{({\ell_\boldh})}(\boldp)\,\delta_{\boldk}^{({\ell_\boldk})}(\boldp)}{r_d(\boldp)}
\\
&=\sum_{\substack{\boldp\not\in I}}
r_d(\boldp)^{-1}
\sum_{\boldh\in I}\delta_{\boldh}^{({\ell_\boldh})}(\boldp)
\sum_{\boldk\in I}\delta_{\boldk}^{({\ell_\boldk})}(\boldp)\\
&
=
\sum_{\substack{\boldp\not\in I}}
r_d(\boldp)^{-1}
\sum_{\ell_1=1}^{L}
\sum_{\boldh\in I_{\ell_1}}\delta_{\boldh}^{({\ell_1})}(\boldp)
\sum_{\ell_2=1}^{L}
\sum_{\boldk\in I_{\ell_2}}\delta_{\boldk}^{({\ell_2})}(\boldp)
\stackrel{\eqref{eq:aliasing_summation_L}}{\le}
L^2\Sigma_{\bcancel{I}\bcancel{I}}.
\end{align*}
\end{proof}
In summary, we achieve the main result of this paper.
\begin{theorem}\label{cor:gen_framework_main_result}
Let $I\subset\Z^d$ and $\Lambda(\boldz_1,M_1,\ldots,\boldz_L,M_L)$ be a rank-1 lattice such that $I=\bigcupdot_{\ell=1}^L I_\ell$, where $I_\ell$ is defined in \eqref{eq:def_Iell}.
Then we have
$$
\sup_{\|f|\mathcal{H}_r(\T^d)\|\le 1}\|f-\operatorname{S}_I^\Lambda f|L_\infty(\T^d)\|\le(L+1)\sqrt{\Sigma_{\bcancel{I}\bcancel{I}}}.$$
\end{theorem}
\begin{proof}
Lemmas \ref{lem:est_sum_cI_I} and \ref{lem:est_sum_I_I} together with \eqref{eq:split_wceinfty} yield the assertion.
\end{proof}
Roughly speaking, the worst case sampling error $\sup_{\|f|\mathcal{H}_r(\T^d)\|\le 1}\|f-\operatorname{S}_I^\Lambda f|L_\infty(\T^d)\|$ for the considered sampling method is bounded by a product of the worst case truncation error $\Sigma_{\bcancel{I}\bcancel{I}}$, cf. \eqref{eq:trunc_error}, and the
number $L$ of rank-1 lattices that needs to be joined in order to observe $I=\bigcupdot I_\ell$.

\begin{remark}\label{rem:averaging_approx_fc}
In \cite{KaVo19}, a slight modification of Algorithm~\ref{alg:compute_S_I_Lambda} is presented as Algorithm~2 and the authors estimate the asymptotic sampling rates for more specific approximation settings. For practical applications, this
algorithm seems preferable since the approximated Fourier coefficients may be computed as an average from several rank-1 lattices, which may prove beneficial in real world applications due to the averaging of the aliasing error. However,
we decided to consider the simpler Algorithm~\ref{alg:compute_S_I_Lambda} in order to avoid unnecessarily complicated calculations that are caused by the averaging process. Nevertheless, the proof strategy presented here succeeds even for
the slightly more complicated algorithm but suffers from additional technical efforts.
\end{remark}

\section{Applications}

\subsection{Tractability}\label{sec:tractability}

The last section clarifies the general framework. In this section, we will
use the results of the general framework in order to treat a specific approximation problem.

\begin{sloppypar}
Similar to the considerations in \cite{KuWaWo09}, we define the reproducing kernel $K_d$ for the weighted
Korobov space $\mathcal{H}_{\alpha,\boldgamma_d}(\T^d)$ with smoothness parameter $\alpha>1$ as
$$
K_d(\boldx,\boldy)=\sum_{\boldh\in\Z^d}\frac{\e^{2\pi\ii\boldh\cdot(\boldx-\boldy)}}{r_d(\alpha,\boldgamma_d,\boldh)},
$$
where for each $d$ the vector $\boldgamma_d=(\gamma_{d,1},\ldots,\gamma_{d,d})$ of positive weights satisfies
\begin{align*}
1\ge\gamma_{d,1}\ge\ldots\ge\gamma_{d,d}>0,
\end{align*}
and the weight function $r_d$ is defined as
$$
r_d(\alpha,\boldgamma_d,\boldh)=\prod_{j=1}^dr(\alpha,\gamma_{d,j},h_j)
$$
with
$
r(\alpha,\gamma_{d,j},h_j)=\max\left(1,\gamma_{d,j}^{-1}\,|h_j|^\alpha\right).
$
Reasonable frequency sets $I$ are constructed by collecting all frequencies $\boldh\in\Z^d$ where the
weight function $r_d(\alpha,\boldgamma_d,\boldh)$ is small, i.e., where the reciprocal $r_d(\alpha,\boldgamma_d,\boldh)^{-1}$
is large, which brings smallest possible worst case truncation errors $\Sigma_{\bcancel{I}\bcancel{I}}$ with respect to the number $|I|$ of excluded frequencies.
We define
\begin{equation}
A_d(N):=\left\{\boldh\in\Z^d\;\colon\; r_d(\alpha,\boldgamma_d,\boldh)\le N\right\},\label{eq:def_AdN}
\end{equation}
which implies 
$
\Sigma_{\bcancel{A_d(N)}\bcancel{A_d(N)}}=\min_{|I|\le|A_d(N)|}\Sigma_{\bcancel{I}\bcancel{I}},
$
i.e., the worst case truncation error $\Sigma_{\bcancel{A_d(N)}\bcancel{A_d(N)}}$ is as small as possible for approximations that
use trigonometric polynomials that are supported by at most $|A_d(N)|$ frequencies.
\end{sloppypar}
We collect some basic facts about the frequency sets $A_d(N)$ from \cite{KuSlWo08}.

\begin{lemma}
For $N\ge 1$, the cardinality of the set $A_d(N)$ is bounded by
\begin{align}
(\gamma_{d,1}N)^{1/\alpha}\le|A_d(N)|\le N^q\prod_{j=1}^d\left(1+2\zeta(\alpha q)\gamma_{d,j}^q\right)\qquad\forall q>\frac{1}{\alpha}.
\label{eq:card_AdN}
\end{align}
Moreover, we observe the set inclusions
\begin{align}
A_d(N)\subset\left[-\floor{(\gamma_{d,1}N)^{1/\alpha}},\floor{(\gamma_{d,1}N)^{1/\alpha}}\right]^d\subset\left[-\frac{|A_d(N)|}{2},\frac{|A_d(N)|}{2}\right]^d
\label{eq:embedding_AdN}
\end{align}
and the upper bound on the worst case truncation error
\begin{align}
\Sigma_{\bcancel{A_d(N)}\bcancel{A_d(N)}}:=\sum_{\boldh\in\Z^d\setminus A_d(N)}r_d(\alpha,\boldgamma_d,\boldh)^{-1}\le\frac{1}{|A_d(N)|^{1/\tau-1}}\frac{\tau}{1-\tau}\prod_{j=1}^d\left(1+2\zeta(\alpha\tau)\gamma_{d,j}^\tau\right)^{1/\tau}
\label{eq:ub_wcterr}
\end{align}
for all $\tau\in(1/\alpha,1)$.

\end{lemma}
\begin{proof}
The proof of the statements \eqref{eq:card_AdN} and \eqref{eq:ub_wcterr} can be found in \cite[Lem. 5 \& 6]{KuSlWo08}.
The set inclusions \eqref{eq:embedding_AdN} can be seen by determining 
$$
r_d(\alpha,\boldgamma_d,\boldh)>N\text{ for }\boldh\in\Z^d\setminus
\left[-\floor{(\gamma_{d,1}N)^{1/\alpha}},\floor{(\gamma_{d,1}N)^{1/\alpha}}\right]^d
$$
and
$$
\floor{(\gamma_{d,1}N)^{1/\alpha}}\le \frac{1+2\floor{(\gamma_{d,1}N)^{1/\alpha}}}{2}=\frac{|A_{d,1}(N)|}{2}\le\frac{|A_{d}(N)|}{2},
$$
where $A_{d,1}(N)$ is the projection of $A_{d}(N)$ to its first dimension.
\end{proof}

In the following, we will apply the results from Section \ref{sec:general_framework} to the specific approximation problem.
To this end, we need to determine multiple rank-1 lattices that fulfill $I=\bigcupdot_{\ell=1}^L I_\ell$, $I_\ell$ as stated in \eqref{eq:def_Iell}. Furthermore, we need upper bounds on the number $L$ of used rank-1 lattices as well as upper bounds
on the number of used sampling values. The next lemma addresses the necessary estimates in full detail.

\begin{lemma}\label{lem:estimate_sampling_set}
Let $N\ge1$ and $A_d(N)$ as stated in \eqref{eq:def_AdN}.
Then there exists a multiple rank-1 lattice
$\Lambda(\boldz_1,M_1,\ldots,\boldz_L,M_L)$ with $L\le\max(3\,\ln |A_d(N)|,1)$ and $M_1=\ldots=M_L\le3\,|A_d(N)|$
that fulfills $I=\bigcupdot_{\ell=1}^dI_\ell$, $I_\ell$
as stated in \eqref{eq:def_Iell}. In particular, the cardinality of $\Lambda(\boldz_1,M_1,\ldots,\boldz_L,M_L)$
is bounded by
\begin{equation*}
2|A_d(N)|<M:=|\Lambda(\boldz_1,M_1,\ldots,\boldz_L,M_L)|< 9\,|A_d(N)|\,\max(\ln |A_d(N)|,1).
\end{equation*}
\end{lemma}

\begin{proof}
We distinguish two cases.
First, we assume $|A_d(N)|=1$. Then, we observe $A_d(N)=\{\boldzero\}$ and the rank-1 lattice $\Lambda(\boldone,3)$ is a multiple rank-1 lattice for which $I_1= A_d(N)$, $L\le 1=\max(3\,\ln |A_d(N)|,1)$, and $M_1=3\le 3|A_d(N)|$ hold.
We estimate
$$2=2|A_d(N)|<3<9\,|A_d(N)|\,\max(\ln |A_d(N)|,1)=9.$$
For the second case, we observe  that $|A_d(N)|\ge 2$ results in $|A_d(N)|\ge 3$ since $\boldzero\in A_d(N)$ and $\boldzero\neq\boldh\in A_d(N)$ implies $-\boldh\in A_d(N)$, $-\boldh\neq\boldh$, due to the symmetry of the weight function $r_d$.
Accordingly, we assume $|A_d(N)|\ge 3$ and apply \cite[Theorems 3.2 \& 3.4]{Kae17} to $I:=A_d(N)$ with $c:=2$ and $1>\delta:=\sqrt{\frac{e}{|A_d(N)|}}>0$.

Thus, we determine
$L=\ceil{2(\ln |I|-\ln \delta)}\le3\,\ln |I|$ and lattice sizes
$M_\ell$ that are prime numbers larger than $2(|I|-1)$ fulfilling the
additional condition
$$M_\ell\in\{M\in\N\colon M\text{ prime with }|\{\boldk\bmod M\colon\boldk\in I\}|=|I|\}.$$
This condition is automatically fulfilled for $M_\ell\ge|I|+1$, $I=A_d(N)$ due to \eqref{eq:embedding_AdN}.
\newline
Due to \cite[Thm. 1.3]{Ba06} there exists at least one prime number $P$ in the interval
$[2\,|I|,3\,|I|]$. We fix this prime number as lattice sizes $M_\ell=P$, $\ell=1,\ldots,L$,
i.e., $M_\ell<3\,|I|$ for each  $\ell=1,\ldots,L$.
\newline
Subsequently, we choose the generating vectors $\boldz_1,\ldots,\boldz_L\in[0,P-1]^d$ uniformly at random.
Then with probability at least $1-\delta>0$ we observe the equality
$\bigcup_{l=1}^L I_\ell'=I$ with
$$
I_\ell':=\left\{\boldk\in I\colon \boldk\cdot\boldz_\ell\not\equiv\boldh\cdot\boldz_\ell\imod{M_\ell}\text{ for all }\boldh\in I\setminus\{\boldk\}\right\}
$$
The simple calculation $I_\ell=I_\ell'\setminus\bigcup_{j=1}^{\ell-1}I_j'$, $\ell=1,\ldots,L$, yields
the disjoint partition $\bigcupdot_{\ell=1}^LI_\ell=I$.
\newline
Since the probability for choosing suitable generating vectors is larger than zero, there exists 
at least one multiple rank-1 lattice with $L\le3\,\ln |I|$ and $M_\ell<3|I|$, $\ell=1,\ldots,L$, that fulfills $\bigcupdot_{\ell=1}^LI_\ell=I$. Accordingly, we estimate 
\begin{align*}
2|A_d(N)|<M_1\le|\Lambda(\boldz_1,M_1,\ldots,\boldz_L,M_L)|<9\,|A_d(N)|\,\max(\ln|A_d(N)|,1)
\end{align*}
for $|A_d(N)|\ge 3$, i.e., for $A_d(N)\neq\{\boldzero\}$.
\end{proof}

\begin{remark}
In the proof of Lemma \ref{lem:estimate_sampling_set}, the failure probability
$\delta=\sqrt{\frac{e}{|A_d(N)|}}$  decreases for increasing cardinalities of the frequency set $A_d(N)$,
i.e., increasing number $N$. For instance, $|A_d(N)|>272$ implies that the failure probability $\delta$ is bounded from above by $1/10$,
i.e., the construction will be successful with a probability of at least $9/10$.
Moreover, the sets $I_\ell$, cf. \eqref{eq:def_Iell}, can be determined in a fast and
simple way, cf. \cite{Kae17} for more details, which allows to check whether the condition
$I=\bigcupdot_{\ell=1}^LI_\ell$ holds. In cases where $I\neq\bigcupdot_{\ell=1}^LI_\ell$, 
one repeats the random choice of the generating vectors and checks the property $I=\bigcupdot_{\ell=1}^LI_\ell$ several times. The corresponding probability that
none of the tested multiple rank-1 lattices ensures $I=\bigcupdot_{\ell=1}^LI_\ell$
decreases exponentially with the number of repetitions. Thus, in practice this strategy inevitably leads to a multiple rank-1 lattice that has the requested property.
Accordingly, we described a practically applicable construction of the sampling sets $\Lambda(\boldz_1,M_1,\ldots,\boldz_L,M_L)$. 
\end{remark}

Now we bring together the relationship of the worst case truncation error, the cardinalities of the frequency sets $A_d(N)$, the estimates of the number $L$ of used rank-1 lattices, and the cardinality of the used sampling set.

\begin{theorem}\label{thm:specific_err_estimate}
The worst case $L_\infty(\T^d)$ sampling error for functions from the Korobov space $\mathcal{H}_{\alpha,\boldgamma_d}(\T^d)$, dimension $d \ge 1$, with smoothness parameter $\alpha>1$ and weights $\boldgamma_d$ is bounded from above by
\begin{align}
\sup_{\|f|\mathcal{H}_{\alpha,\boldgamma_d}(\T^d)\|\le 1}&\|f-\operatorname{S}_{A_d(N)}^\Lambda f|L_\infty(\T^d)\|
\nonumber\\
&< 4\,3^{1/\tau-1}(\ln M)^{\frac{1+\tau}{2\tau}}M^{\frac{\tau-1}{2\tau}}\sqrt{\frac{\tau}{1-\tau}}\prod_{j=1}^d\left(1+2\zeta(\alpha\tau)\gamma_{d,j}^\tau\right)^{\frac{1}{2\tau}}\nonumber\\
&\le 4\,3^{1/\tau-1}\delta^{-\frac{1+\tau}{2\tau}}M^{\frac{1+\delta}{2}-\frac{1-\delta}{2\tau}}
\sqrt{\frac{\tau}{1-\tau}}\prod_{j=1}^d\left(1+2\zeta(\alpha\tau)\gamma_{d,j}^\tau\right)^{\frac{1}{2\tau}},
\label{eq:err_estimate_arbitrary_delta_tau}
\end{align}
when using the multiple rank-1 lattices established in Lemma \ref{lem:estimate_sampling_set}.
The estimates in \eqref{eq:err_estimate_arbitrary_delta_tau} hold for all parameters $\tau$ and $\delta$ in their ranges $\tau\in(1/\alpha,1)$ and $\delta\in(0,1)$
and the number $M:=|\Lambda(\boldz_1,M_1,\ldots,\boldz_L,M_L)|$ is the total number of used sampling values.
\end{theorem}

\begin{proof}
Due to Theorem \ref{cor:gen_framework_main_result} and \eqref{eq:ub_wcterr} we observe
\begin{align*}
\sup_{\|f|\mathcal{H}_{\alpha,\boldgamma_d}(\T^d)\|\le 1}&\|f-\operatorname{S}_{A_d(N)}^\Lambda f|L_\infty(\T^d)\|^2\le (L+1)^2\Sigma_{\bcancel{I}\bcancel{I}}\\
&\le (L+1)^2
\frac{1}{|A_d(N)|^{1/\tau-1}}\frac{\tau}{1-\tau}\prod_{j=1}^d\left(1+2\zeta(\alpha\tau)\gamma_{d,j}^\tau\right)^{1/\tau}.
\end{align*}
For computing approximations, we use the multiple rank-1 lattices from Lemma \ref{lem:estimate_sampling_set} and conclude
\begin{align*}
3&\le M,\\
\max(\ln |A_d(N)|,1)&<\ln M,\\
\frac{1}{|A_d(N)|}&<\frac{9\max(\ln |A_d(N)|,1)}{M}<\frac{9\,\ln M}{M},\\
L+1&\le 4\max(\ln|A_d(N)|,1),
\end{align*}
which yields
\begin{align*}
\sup_{\|f|\mathcal{H}_{\alpha,\boldgamma_d}(\T^d)\|\le 1}&\|f-\operatorname{S}_{A_d(N)}^\Lambda f|L_\infty(\T^d)\|^2\\[-1.5em]
&\le 16\max(\ln |A_d(N)|,1)^2
\frac{1}{|A_d(N)|^{1/\tau-1}}\overbrace{\frac{\tau}{1-\tau}\prod_{j=1}^d\left(1+2\zeta(\alpha\tau)\gamma_{d,j}^\tau\right)^{1/\tau}}^{=:c_{\alpha,d,\tau}}\\
&<16\,c_{\alpha,d,\tau}\,(\ln M)^2\left(\frac{9\,\ln M}{M}\right)^{1/\tau-1}
=16\,9^{1/\tau-1}c_{\alpha,d,\tau}\frac{(\ln M)^{1/\tau+1}}{M^{1/\tau-1}}.
\end{align*}
In order to avoid the logarithmic terms, we exploit $\ln x\le x^\delta/\delta$ for all $\delta\in(0,1)$ and we get
\begin{align*}
\sup_{\|f|\mathcal{H}_{\alpha,\boldgamma_d}(\T^d)\|\le 1}\|f-\operatorname{S}_{A_d(N)}^\Lambda f|L_\infty(\T^d)\|^2
&<16\,9^{1/\tau-1}c_{\alpha,d,\tau}\delta^{-1/\tau-1}\frac{1}{M^{1/\tau-1-\delta(1/\tau+1)}}.
\end{align*}
\end{proof}

The last theorem immediately raises the question of how to choose proper parameters $\delta$ and $\tau$ such that
one can reach best rates of convergence with respect to $M$ in \eqref{eq:err_estimate_arbitrary_delta_tau}.
The answer is given by the next corollary which also states the occurring constants in detail.

\begin{corollary}\label{cor:korobov_main}
With the requirements of Theorem \ref{thm:specific_err_estimate}, for each
$t\in(0,\frac{\tilde{\alpha}-1}{2})$, $1<\tilde{\alpha}\le\alpha$, there exist $\delta:=\delta(\tilde{\alpha},t)\in(0,1)$ and $\tau:=\tau(\tilde{\alpha},t)\in(1/\tilde{\alpha},1)\subset(1/\alpha,1)$ such that
there exists a constant
\begin{align*}
c_{\alpha,\tilde{\alpha},t,d}&:=4\,3^{1/\tau-1}\delta^{-\frac{1+\tau}{2\tau}}
\sqrt{\frac{\tau}{1-\tau}}\prod_{j=1}^d\left(1+2\zeta(\alpha\tau)\gamma_{d,j}^\tau\right)^{\frac{1}{2\tau}}\\
&<4\,3^{\tilde{\alpha}-1}\delta^{-\frac{\tilde{\alpha}+1}{2}}\sqrt{\frac{2}{\tilde{\alpha}-1}} \prod_{j=1}^d\left(1+2\zeta(\alpha\tau)\gamma_{d,j}^\tau\right)^{\frac{1}{2\tau}}
\end{align*}
which allows for the estimate
$$
\sup_{\|f|\mathcal{H}_{\alpha,\boldgamma_d}(\T^d)\|\le 1}\|f-\operatorname{S}_{A_d(N)}^\Lambda f|L_\infty(\T^d)\|
<c_{\alpha,\tilde{\alpha},t,d}M^{-t}.
$$
Here $\delta(\tilde{\alpha},t)$ and $\tau(\tilde{\alpha},t)$ can be chosen as 
\begin{align}
\delta(\tilde{\alpha},t)&:=\frac{2+\tilde{\alpha}}{2}-\sqrt{\frac{2+\tilde{\alpha}}{2}-\tilde{\alpha}+1+2t}\nonumber\\
\tau(\tilde{\alpha},t)&:=(\tilde{\alpha}-\delta(\tilde{\alpha},t))^{-1}.\label{eq:definition_suitable_tau}
\end{align}

\end{corollary}
\begin{proof}
For fixed $t$ we determine $\varepsilon:=\tilde{\alpha}-1-2t$ such that $t=\frac{\tilde{\alpha}-1-\varepsilon}{2}$, i.e., $\varepsilon\in(0,\tilde{\alpha}-1)\subset(0,\alpha-1)$. Moreover, we fix
$$
\delta:=\frac{2+\tilde{\alpha}}{2}-\sqrt{\left(\frac{2+\tilde{\alpha}}{2}\right)^2-\varepsilon},
$$
which implies $\delta\in(0,1)$ since
\begin{align*}
0<\delta:=\frac{2+\tilde{\alpha}}{2}-\sqrt{\left(\frac{2+\tilde{\alpha}}{2}\right)^2-\varepsilon}
&<\frac{2+\tilde{\alpha}}{2}-\sqrt{\left(\frac{2+\tilde{\alpha}}{2}\right)^2-\tilde{\alpha}+1}\\
&=\frac{1}{2}(\tilde{\alpha}+2-\sqrt{\tilde{\alpha}^2+8})
<\frac{1}{2}(\tilde{\alpha}+2-\sqrt{\tilde{\alpha}^2})=1,
\end{align*}
and we set $\tau=\frac{1}{\tilde{\alpha}-\delta}$ which implies $\tau>\frac{1}{\tilde{\alpha}}$ and
with
\begin{align*}
\delta<\frac{1}{2}(\tilde{\alpha}+2-\sqrt{\tilde{\alpha}^2+8})\le \frac{1}{2}(\tilde{\alpha}+2-3)=\frac{1}{2}(\tilde{\alpha}-1)
\end{align*}
we estimate $\tau<\frac{1}{\tilde{\alpha}-\frac{\tilde{\alpha}-1}{2}}=\frac{1}{\tilde{\alpha}}+\frac{\tilde{\alpha}-1}{\tilde{\alpha}(\tilde{\alpha}+1)}=\frac{2}{\tilde{\alpha}+1}<1$, i.e., $\tau\in(1/\tilde{\alpha},1)\subset(1/\alpha,1)$.
\newline
Due to Theorem \ref{thm:specific_err_estimate}, we have
\begin{align*}
\sup_{\|f|\mathcal{H}_{\alpha,\boldgamma_d}(\T^d)\|\le 1}&\|f-\operatorname{S}_{A_d(N)}^\Lambda f|L_\infty(\T^d)\|\\
&<4\,3^{1/\tau-1}\delta^{-\frac{1+\tau}{2\tau}}M^{\frac{1+\delta}{2}-\frac{1-\delta}{2\tau}}
\sqrt{\frac{\tau}{1-\tau}}\prod_{j=1}^d\left(1+2\zeta(\alpha\tau)\gamma_{d,j}^\tau\right)^{\frac{1}{2\tau}}.
\intertext{Since $\tau$ and $\delta$ are fixed and in the right range for fixed $t$ and $\alpha$, we obtain}
\sup_{\|f|\mathcal{H}_{\alpha,\boldgamma_d}(\T^d)\|\le 1}&\|f-\operatorname{S}_{A_d(N)}^\Lambda f|L_\infty(\T^d)\|
<
c_{\alpha,\tilde{\alpha},t,d}M^{\frac{1+\delta}{2}-\frac{1-\delta}{2\tau}},
\end{align*}
where $c_{\alpha,\tilde{\alpha},t,d}$ depends on $d$, $\boldgamma_d$, $\alpha$ as well as $\tilde{\alpha}$ and $t$ since $\delta$ and $\tau$ depend only on $\tilde{\alpha}$ and $\varepsilon=\tilde{\alpha}-1-2t$.
We verify the main rate in $M$
\begin{align*}
\frac{1+\delta}{2}-\frac{1-\delta}{2\tau}&=\frac{1+\delta}{2}-\frac{(1-\delta)(\tilde{\alpha}-\delta)}{2}=\frac{1}{2}\left(1-\tilde{\alpha}+2\delta+\tilde{\alpha}\delta-\delta^2\right)=-t
\end{align*}
which is caused by the choice of $\delta$ such that
$\varepsilon=2\delta+\tilde{\alpha}\delta-\delta^2$.
Furthermore, the additional estimate on $c_{\alpha,\tilde{\alpha},t,d}$ holds:
\begin{align*}
c_{\alpha,\tilde{\alpha},t,d}&<4\,3^{\tilde{\alpha}-1}\delta^{-\frac{\tilde{\alpha}+1}{2}}\sqrt{\frac{2}{\tilde{\alpha}-1}} \prod_{j=1}^d\left(1+2\zeta(\alpha\tau)\gamma_{d,j}^\tau\right)^{\frac{1}{2\tau}},
\end{align*}
due to the inequalities $1/\tau<\tilde{\alpha}$, $(1+\tau)/(2\tau)<(\tilde{\alpha}+1)/2$, and $\tau/(1-\tau)<2/(\tilde{\alpha}-1)$.
\end{proof}

\begin{remark}
The spaces $\mathcal{H}_{\alpha,\boldgamma_d}(\T^d)$ are function spaces with dominating mixed smoothness.
Fixing the dimension $d$ and considering $\tilde{\alpha}=\alpha$ in Corollary~\ref{cor:korobov_main} yields
a general approximation result for those function spaces. We observe an asymptotic behavior of the sampling error 
$\sup_{\|f|\mathcal{H}_{\alpha,\boldgamma_d}(\T^d)\|\le 1}\|f-\operatorname{S}_I^\Lambda f|L_\infty(\T^d)\|$ that is bounded by $c_{\alpha,t,d}M^{-t}$ for each $t<\frac{\alpha-1}{2}$, i.e., the sampling rate is almost optimal.
More detailed, but for the following purposes less convenient estimates on the cardinalities $|A_d(N)|$ of the frequency sets $A_d(N)$ and the worst case truncation errors $\Sigma_{\bcancel{A_d(N)}\bcancel{A_d(N)}}$
will lead to asymptotic estimates on the sampling errors that are optimal with respect to the exponent on the number $M$ of sampling values, i.e., one achieves 
$$\sup_{\|f|\mathcal{H}_{\alpha,\boldgamma_d}(\T^d)\|\le 1}\|f-\operatorname{S}_I^\Lambda f|L_\infty(\T^d)\|\le C M^{-\frac{\alpha-1}{2}}\log^{b}{M},$$
which even improves the results in \cite{KaVo19}. However, the exponent $b$ at the logarithmic term is linear in the dimension $d$ and not optimal.
\end{remark}

Up to now, the constants $c_{\alpha,\tilde{\alpha},t,d}$ heavily depend on the dimension. The following corollaries categorize the properties of the weights $\boldgamma_d$ that lead to bounds on the constants that are independent on the dimension $d$ or polynomials in $d$. 

\begin{corollary}\label{cor:kor_strong_tractability}
When choosing $\tilde{\alpha}=\min\{\alpha,1/s_\boldgamma\}$ in Corollary~\ref{cor:korobov_main} and assuming that
$$
s_\boldgamma:=\inf\left\{s \ge 0\,\colon \sup_{d\ge 1}\sum_{j=1}^d\gamma_{d,j}^s<\infty\right\}<1,
$$
holds,
the $L_\infty(\T^d)$ worst case sampling error is bounded by terms that do not depend on the dimension~$d$, i.e.,
$$
\sup_{\|f|\mathcal{H}_{\alpha,\boldgamma_d}(\T^d)\|\le 1}\|f-\operatorname{S}_{A_d(N)}^\Lambda f|L_\infty(\T^d)\|
<c_{\alpha,\tilde{\alpha},t}M^{-t},
$$
we observe strong tractability for each $t\in(0,\frac{\tilde{\alpha}-1}{2})$.
\end{corollary}
\begin{proof}
Since $s_\boldgamma<1$ , we observe $\tilde{\alpha}>1$ and we apply Corollary \ref{cor:korobov_main}. For fixed $\tau\in(1/\tilde{\alpha},1)$, the products
$$\prod_{j=1}^d\left(1+2\zeta(\alpha\tau)\gamma_{d,j}^\tau\right)^{\frac{1}{2\tau}}
\le\e^{\frac{2\zeta(\alpha\tau)\sum_{j=1}^{d}\gamma_{d,j}^\tau}{2\tau}}
\le C^{\prod}_{\alpha,\tilde{\alpha},t}<\infty$$
are bounded without dependence on the dimension $d$.
Hence, we observe $$
c_{\alpha,\tilde{\alpha},t,d}<4\,3^{\tilde{\alpha}-1}\delta^{-\frac{\tilde{\alpha}+1}{2}}\sqrt{\frac{2}{\tilde{\alpha}-1}}C^{\prod}_{\alpha,\tilde{\alpha},t}=:c_{\alpha,\tilde{\alpha},t}.
$$
\end{proof}

\begin{corollary}\label{cor:kor_tractability}
When choosing $\tilde{\alpha}=\min\{\alpha,1/t_\boldgamma\}$ in Corollary~\ref{cor:korobov_main} and assuming that
$$
t_\boldgamma:=\inf\left\{t \ge 0\,\colon \sup_{d\ge 1}\frac{\sum_{j=1}^d\gamma_{d,j}^t}{\ln(d+1)}<\infty\right\}<1,
$$
the $L_\infty(\T^d)$ worst case sampling error is bounded by terms that depend polynomially on the dimension~$d$, i.e.,
$$
\sup_{\|f|\mathcal{H}_{\alpha,\boldgamma_d}(\T^d)\|\le 1}\|f-\operatorname{S}_{A_d(N)}^\Lambda f|L_\infty(\T^d)\|
<c_{\alpha,\tilde{\alpha},t}\,d^{\beta(\alpha,\tilde{\alpha},t)}M^{-t},
$$
we observe tractability for each $t\in(0,\frac{\tilde{\alpha}-1}{2})$.
\end{corollary}
\begin{proof}
For fixed $t\in(0,\frac{\tilde{\alpha}-1}{2})$, we determine $\tau>1/\tilde{\alpha}\ge t_\boldgamma$ as stated in \eqref{eq:definition_suitable_tau}. We observe 
$$
\sup_{d\ge 1}\frac{\sum_{j=1}^d\gamma_{d,j}^\tau}{\ln(d+1)}=S_{\tilde{\alpha},t}<\infty
$$
and estimate
$$\prod_{j=1}^d\left(1+2\zeta(\alpha\tau)\gamma_{d,j}^\tau\right)^{\frac{1}{2\tau}}\le \e^{\,\frac{\zeta(\alpha\tau)\,S_{\tilde{\alpha},t}}{\tau}\ln(2d)}=(2d)^{\frac{\zeta(\alpha\tau)\,S_{\tilde{\alpha},t}}{\tau}}\le \tilde{c}_{\alpha,\tilde{\alpha},t}d^{\beta(\alpha,\tilde{\alpha},t)},$$
where $\beta(\alpha,\tilde{\alpha},t):=\ceil{\frac{\zeta(\alpha\tau)\,S_{\tilde{\alpha},t}}{\tau}}$
does not depend on the dimension $d$.
Hence, we observe for $c_{\alpha,\tilde{\alpha},t,d}$ from Theorem \ref{thm:specific_err_estimate}
$$
c_{\alpha,\tilde{\alpha},t,d}<4\,3^{\tilde{\alpha}-1}\delta^{-\frac{\tilde{\alpha}+1}{2}}\sqrt{\frac{2}{\tilde{\alpha}-1}}\tilde{c}_{\alpha,\tilde{\alpha},t}d^{\beta(\alpha,\tilde{\alpha},t)}=:c_{\alpha,\tilde{\alpha},t}d^{\beta(\alpha,\tilde{\alpha},t)},
$$
which is actually a polynomial in the dimension $d$. Again, we stress that $\delta$ and $\tau$ are completely determined in terms of
$\tilde{\alpha}$ and $t$, cf.~\eqref{eq:definition_suitable_tau}, and thus the constant as well as the exponent $\beta$ on the right hand side of the last inequality only depend on $\alpha$, $\tilde{\alpha}$, and $t$.
\end{proof}

\begin{remark}
We stress the fact that the restrictions $s_\boldgamma<1$ and $t_\gamma<1$ in Corollaries~\ref{cor:kor_strong_tractability} and~\ref{cor:kor_tractability} are necessary in order to achieve an admissible interval for the choice of $\tau$ in Corollary~\ref{cor:korobov_main}, i.e., these restrictions are caused by the used proof technique.
Nevertheless, these restrictions coincide with those that are stated in \cite[Theorem~11]{KuWaWo08}, where
the authors proved that (strong) tractability can not hold for $L_\infty(\T^d)$ approximation for
$t_\boldgamma\ge 1$ ($s_\boldgamma\ge 1$) -- even in a more general setting.
Hence, these requirements on the summability of the weight sequence $\boldgamma$ are natural ones
and do not additionally restrict the function spaces for which the statements of Corollaries~\ref{cor:kor_strong_tractability} and~\ref{cor:kor_tractability}
hold.
\end{remark}

\begin{remark}\label{rem:tractability}
Our sampling method allows for error estimates that reveal (strong) tractability and the exponent at the number of sampling values $M$ is $-t$ with the single restriction that $t<\frac{\tilde{\alpha}-1}{2}$ holds. From~\cite{KuWaWo08} one knows that tractability may hold for exponents $t$ at most $\frac{\tilde{\alpha}-1}{2}$. Thus, the presented sampling method using samples along multiple rank-1 lattices is nearly optimal for Korobov spaces of the considered type even with respect to tractability. Moreover, the presented result improves known tractability results for lattice algorithms, cf.~\cite{KuWaWo09}. Up to now, tractability of the $L_\infty(\T^d)$ approximation in combination with constructive methods was investigated for single rank-1 lattices. The best known lower bounds on $t$, where $-t$ is the exponent at the number $M$ of used sampling values, suffer from factors of $1/2$
and less than $2/3$ for $\alpha\in(1,2]$ and $\alpha>2$, respectively,
compared to the best possible exponent $\frac{\tilde{\alpha}-1}{2}$ for general linear information.

In addition, our sampling approach improves more general tractability results from~\cite{KuWaWo09b}. Therein, it is shown that there exist sampling sets such that the $L_\infty(\T^d)$ approximation is tractable. The corresponding exponents $-t$ at $M$ are bounded by $\frac{\tilde{\alpha}-1}{2}\frac{1}{1+1/\tilde{\alpha}}\le t \le \frac{\tilde{\alpha}-1}{2}$, where the lower bound is the crucial point that needs to be improved in order to verify better asymptotic tractability behavior of sampling methods. Sampling along multiple rank\mbox{-}1 lattices provides exactly this improvement, i.e., the exponent $t$ at $1/M$  can be arbitrarily close to the upper bound $\frac{\tilde{\alpha}-1}{2}$, which stems from considerations of approximating functions using general linear information, cf.~\cite{KuWaWo08}. At this point, we again stress the fact that --  in contrast to the considerations in~\cite{KuWaWo09b} -- the generation of the sampling sets for the multiple rank-1 lattice approach is completely constructive.
\end{remark}

\subsection{Application to sampling numbers}\label{sec:app_samp_numbers}

A general concept to describe the approximability of a bounded linear operator 
$\operatorname{T} \colon X\to Y$, where $X$ and $Y$ are Banach spaces, is the definition of so called
approximation numbers
$$
a_n(\operatorname{T})=\inf\{\|\operatorname{T}-\operatorname{A}\|\,\colon\, \operatorname{rank}\operatorname{A}<n\},
$$
which is the optimal error of approximating the operator $\operatorname{T}$ by operators of rank less than $n$.

The setting of our particular interest are the $n$-th approximation numbers of the identity operator $\operatorname{I}_d$
which maps from specific Hilbert spaces $\mathcal{H}_r(\T^d)$ of (smooth and continuous) periodic functions to $L_\infty(\T^d)$.
Due to \cite[Theorem 3.4]{CoKueSi16} the associated approximation numbers  are determined as
\vspace*{-0.5em}
\begin{align}a_n(\operatorname{I}_d\,\colon\, \mathcal{H}_r(\T^d)\rightarrow L_\infty(\T^d))=\left(\sum_{j=n}^\infty r_d(\boldk_j)^{-1}\right)^{1/2},\label{eq:an_repro_kernel}
\end{align}
where $\{\boldk_j\}_{j=1}^\infty$ is a rearrangement of all elements within $\Z^d$ such that
the sequence $\{r_d(\boldk_j)\}_{j=1}^\infty$ is a non-decreasing sequence, i.e., $r_d(\boldk_j)\le r_d(\boldk_{j+1})$ holds for all $j\in\N$.
One interpretation of the approximation numbers $a_n(\operatorname{I}_d\colon \mathcal{H}_r(\T^d)\rightarrow L_\infty(\T^d))$ is the following:
The linear operator of rank less than $n$ yielding the best possible worst case error, i.e., that achieves
the approximation number $a_n$, is a mapping of the function $f$ to its exact Fourier partial sum
$S_{I^n}f:=\sum_{\boldk\in {I^n}}c_\boldk(f)\e^{2\pi\ii\boldk\cdot\circ}$, where $I^n$, $n\in\N$, is defined by
\vspace*{-0.5em}
\begin{align*}
|I^n|=n-1\quad \textnormal{and}\quad I^n:=\left\{\boldk\in\Z^d\colon r_d(\boldk)\le r_d(\boldh) \textnormal{ for all }\boldk\in I^n \textnormal{ and all }\boldh\in\Z^d\setminus I^n\right\}.
\label{eq:def_I_appr_num}
\end{align*}
Note that $I^n$ as well as the numbering of $\{\boldk_j\}_{j=1}^\infty$ is not uniquely defined.
Moreover, for sampling operators, there exists the concept of sampling numbers, which classifies the quality
of approximations of functions based on samples with respect to the number of samples.
In our specific setting, the corresponding sampling numbers are defined by
\vspace*{-0.5em}
\begin{align*}
g_M(\mathcal{H}_r(\T^d)&,L_\infty(\T^d)):=\\
&\inf_{\mathcal{X},|\mathcal{X}|\le M}\inf_{\operatorname{A}\,\colon\,\C^{|\mathcal{X}|}\to L_\infty(\T^d)}
\sup_{\|f|\mathcal{H}_r(\T^d)\|\le 1}\left\|f-\operatorname{A}\left(\{f(\boldx)\}_{\boldx\in\mathcal{X}}\right)|L_\infty(\T^d)\right\|,
\end{align*}
which can be described as the best possible worst case error of approximating a function $f$ that belongs to the unit ball of the space $\mathcal{H}_r(\T^d)$ by the best possible sampling strategy using not more than $M$ sampling values.

Clearly, the worst case sampling error which is determined by the specific sampling method presented in this paper yields an upper bound
on $g_M(\mathcal{H}_r(\T^d),L_\infty(\T^d))$.
Moreover, the sequence $g_M(\mathcal{H}_r(\T^d),L_\infty(\T^d))$ is nonincreasing in $M$. Taking this into account, 
we observe the following statement.

\begin{theorem}\label{thm:samp_app}
Let $\mathcal{H}_r(\T^d)$ be a reproducing kernel Hilbert space as defined in Section~\ref{sec:repro_kernel_hilbert_space}. In addition, we assume that
$I^{n}\subset[-n+1,n-1]^d$ holds for all $n\in\N$. Then, we estimate
\begin{align*}
a_{M}(\mathcal{H}_r(\T^d),L_\infty(\T^d))\le
g_{M}(\mathcal{H}_r(\T^d),L_\infty(\T^d))
\lesssim
\frac{a_{M/\log{M}}(\operatorname{I}_d\,\colon\, \mathcal{H}_r(\T^d)\rightarrow L_\infty(\T^d))}{(\ln{M})^{-1}}.
\end{align*}
\end{theorem}

\begin{proof}
We follow the argumentation
of the proof of Lemma \ref{lem:estimate_sampling_set}. Due to the assumptions on $I^{n}$,
there exists a multiple rank-1 lattice $\Lambda(\boldz_1,M_1,\ldots, \boldz_L,M_L)$
with $L\le \max(3\ln|I^n|,1)$ and $M_\ell\le 3|I^n|$, $\ell=1,\ldots,L$, that fulfills \eqref{eq:reco_prop}.
We apply Theorem~\ref{cor:gen_framework_main_result} and obtain
\begin{eqnarray}
&&\hspace*{-6em}g_{9(n-1)\max(\ln(n-1),1)}(\mathcal{H}_r(\T^d),L_\infty(\T^d))\le
g_{|\Lambda|}(\mathcal{H}_r(\T^d),L_\infty(\T^d))\nonumber\\
&\le&
 \sup_{\|f|\mathcal{H}_r(\T^d)\|\le 1}\left\|f-\operatorname{S}_{I^n}^\Lambda f|L_\infty(\T^d)\right\|
 \overset{\text{Thm.~\ref{cor:gen_framework_main_result}}}{\le} (L+1)\left(\sum_{j=n}^\infty r_d(\boldk_j)^{-1}\right)^{1/2}\nonumber\\
&\overset{\eqref{eq:an_repro_kernel}}{\le}&
\max(3\,\ln(n-1)+1,2)\;a_{n}(\operatorname{I}_d\,\colon\, \mathcal{H}_r(\T^d)\rightarrow L_\infty(\T^d)),\label{eq:sampl_num_appr_num}
\end{eqnarray}
\end{proof}
The statement of the last theorem can be generalized using the following constraints:
\begin{enumerate}
\item $\mathcal{H}_r(\T^d)$ is a reproducing kernel Hilbert space as defined in Section~\ref{sec:repro_kernel_hilbert_space}, %
\item \label{constraint:n_I^n}$\exists c>0$ such that for each $n\ge 2$ the subset relation $I^n \subset [-c\,n,c\,n]^d$ holds.
\end{enumerate}
One may interpret \eqref{eq:sampl_num_appr_num} in the following way:\newline
The worst case error measured in the $L_\infty(\T^d)$ norm of the approximation of functions belonging to the reproducing kernel Hilbert space $\mathcal{H}_r(\T^d)$ by a trigonometric polynomial that is supported by the frequency set $I^n$ containing $n-1$ frequencies is best possible, if one chooses the exact Fourier partial sum supported by frequencies with smallest possible weights $r_d( \boldk )$ as approximant. The approximation number $a_n(\operatorname{I}_d\,\colon\, \mathcal{H}_r(\T^d)\rightarrow L_\infty(\T^d))$ specifies the corresponding error.
A suitable approximation of the aforementioned Fourier partial sum can be
computed from a number of $|\Lambda|<9(n-1)\max(\ln(n-1),1)$ samples, where the corresponding worst case sampling error 
$g_{|\Lambda|}(\mathcal{H}_r(\T^d),L_\infty(\T^d))$
is bounded from above
by the approximation number $a_n(\operatorname{I}_d\,\colon\, \mathcal{H}_r(\T^d)\rightarrow L_\infty(\T^d))$ that is related to the frequency set $I^n$ times a logarithmic factor in the cardinality $n-1$ of $I^n$ and a constant less than four.

However, this estimate holds even for small values of $n$. Note that the sequence $\{a_n\}_{n\in\N}$ is monotonically decreasing.
Assuming that  $\{a_n\}_{n\in\N}$ decreases faster than $(3\,\ln(n-1)+1)^{-1}$, we observe a reasonably practicable
sampling method that allows for treating approximation problems even in pre-asymptotic settings, which
are more likely to be the rule than the exception -- at least in dimensions $d>3$.

\begin{remark}\label{rem:approx_num}
The very general result of this section is subject to extremely weak constraints.
E.g., constraint \ref{constraint:n_I^n} is actually a restriction on the weight function
$r_d$. In any cases where the weight function $r_d$ generates index sets $I^n$, where $I^2=\{\boldzero\}$
and each $I^n$ is downward closed, i.e., all $I^n$'s are lower sets, we certainly observe $I^n\subset[-n+1,n-1]^d$.
Furthermore, suitable reconstructing rank-1 lattices $\Lambda$, can be determined
by possibly multiple, but only a few number of applications of \cite[Algorithm~3]{Kae17}.
\end{remark}

\section{Conclusions}
We analyzed a recently developed sampling strategy for multivariate periodic functions
and determined upper bounds on the corresponding worst case sampling error measured in the $L_\infty(\T^d)$ norm.
The strategy is to sample functions from reproducing kernel Hilbert spaces
along multiple rank-1 lattices and approximate the function by approximating a suitable set
of Fourier coefficients and assembling corresponding approximating trigonometric polynomials.
It turns out that the considered worst case approximation errors are almost optimal with respect to the
space of trigonometric polynomials that is used for computing the approximant, cf. Theorem~\ref{cor:gen_framework_main_result}.
The crucial assumption on the used multiple rank-1 lattice is the specific reconstruction property $I=\bigcupdot_{\ell=1}^{L}I_\ell$, $I_\ell$ as stated in~\eqref{eq:def_Iell}, cf.~\cite{Kae17}.

Under certain mild assumptions, cf. Section~\ref{sec:app_samp_numbers},
there exist multiple rank-1 lattices that fulfill the reconstruction property
and, moreover,  their number of sampling nodes is bounded from above by terms
$\lesssim |I|\log|I|$. This yields nearly optimal estimates of sampling numbers in terms of approximation numbers, cf. Theorem~\ref{thm:samp_app}.
A further application leads to significantly improved tractability results for sampling methods, cf. Section~\ref{sec:tractability}.
 
Again, we stress on the fact that all suggested algorithms, i.e.,
\begin{itemize}
\item the construction of the sampling sets,
\item the verification of the reconstruction property, and
\item the corresponding discrete Fourier transform 
\end{itemize}
are efficiently implementable. The algorithmic complexity of each
of them is bounded from above by a product of linear terms in the cardinality $|I|$,
linear terms in the dimension $d$, and a few logarithmic factors in $|I|$, cf. Algorithm~\ref{alg:compute_S_I_Lambda} as well as
\cite[Algorithm~3]{Kae17}.

\section*{Acknowledgments}
The author thanks Thomas K\"uhn and Winfried Sickel for pointing out the close connection of the main result of this paper to
their results on approximation numbers in \cite{CoKueSi16} and for the valuable discussions on that topic at the workshop "Challenges in optimal recovery and hyperbolic cross approximation" at the INI in Cambridge, 2019.
The author gratefully acknowledges the funding by the Deutsche Forschungsgemeinschaft (DFG, German Research Foundation, project number 380648269).

\end{document}